\RequirePackage{fix-cm}
\documentclass[smallextended]{svjour3}       
\smartqed  
\usepackage{graphicx}
\usepackage{amssymb}
\usepackage{amsmath}
\usepackage{color}

\begin{document}

\title{Characterization of manifolds of constant curvature by spherical curves\thanks{This is a pre-print of an article published in Annali di Matematica Pura ed Applicata. The final authenticated version is available online at https:$//$doi.org$/$10.1007$/$s10231-019-00874-5} }

\author{Luiz C. B. da Silva         \and
        Jos\'e D. da Silva 
}


\institute{Da Silva, L. C. B. \at Department of Physics of Complex Systems,\\ 
Weizmann Institute of Science\\
             Rehovot 7610001, Israel\\
         \email{luiz.da-silva@weizmann.ac.il}
           \and 
           Da Silva, J. D. \at
             Departamento de Matem\'atica,\\
             Universidade Federal Rural de Pernambuco, \\
         52171-900, Recife, Pernambuco, Brazil\\
         \email{jose.dsilva@ufrpe.br}
}
\date{Received: 29 Oct 2018 / Accepted: 03 Jun 2019}

\maketitle

\begin{abstract}
It is known that the so-called rotation minimizing (RM) frames allow for a simple and elegant characterization of geodesic spherical curves in Euclidean, hyperbolic, and spherical spaces through a certain linear equation involving the coefficients that dictate the RM frame motion [L.C.B. da Silva and J.D. da Silva, Mediterr. J. Math. \textbf{15}, 70 (2018)]. Here, we shall prove the {converse}, i.e., we show that if all geodesic spherical curves on a Riemannian manifold are characterized {by} a certain linear equation, then all the geodesic spheres with a sufficiently small radius are totally umbilical and, consequently, the given manifold has constant sectional curvature. We also furnish two other characterizations in terms of (i) an inequality involving the mean curvature of a geodesic sphere and the curvature function of their curves and (ii) the vanishing of the total torsion of closed spherical curves in the case of three-dimensional manifolds. Finally, we also show that the same results are valid for semi-Riemannian manifolds of constant sectional curvature.
\end{abstract}

\section{Introduction}
\label{intro}

The study of (geodesic) spheres plays a fundamental role in the theory of Riemannian manifolds due to their simplicity and remarkable properties. Amazingly, in some contexts the behavior of spheres suffices to describe the geometry of the ambient manifold as a whole. More precisely, it is known to be possible to characterize space forms, i.e., Riemannian manifolds with constant sectional curvature, in terms of their geodesic spheres. Indeed, we can mention the characterizations of space forms by employing the geodesics \cite{AdachiBAMS2000}, the mean curvature and volume function \cite{ChenCrelle1981}, or {the} umbilicity \cite{VanheckePRSocEdinburgh1979} of geodesic spheres.  The main goal of this work is to add new items to this list by investigating the behavior of geodesic spherical curves with the help of the so-called rotation minimizing (RM) frames. (Concerning the geometry of curves, we may also mention the characterization of space forms in terms of {the validity of} the Frenet theorem \cite{CastrillonLopezAM2015}, i.e., the 1-1 correspondence between curves and curvature/torsion up to isometries.)

The characterization of spherical curves in Euclidean space is a relatively well known subject. On the other hand, in a recent work \cite{daSilvaMJM2018}, the present authors went beyond the Euclidean space and were able to characterize geodesic spherical curves in $\mathbb{H}^{m+1}(r)$ and $\mathbb{S}^{m+1}(r)$: any spherical curve $\alpha$ in one of these spaces is associated with an equation of the form
\begin{equation}
\sum_{i=1}^ma_i\kappa_i(s)-\lambda=0,\label{eq::SimpleCurvesEq}
\end{equation}
where $a_1,\dots,a_m$ are constants, $\kappa_1,\dots,\kappa_m$ are the curvatures with respect to a rotation minimizing frame $\{\mathbf{t},\mathbf{n}_1,\dots,\mathbf{n}_m\}$ along $\alpha(s)$, and $\lambda$ is a constant whose value depends on the radii of the geodesic sphere and ambient space. 

A fundamental step in the proof of Eq. (\ref{eq::SimpleCurvesEq}) consisted in the observation that, along $\alpha$, the unit normal $\xi$ of a geodesic sphere $G(p,R)$ can be written as $
\xi(s) = \sum_{i=1}^ma_i\mathbf{n}_i(s),\,a_i\,\mbox{ constant}$, and in addition that $\xi(s)$ minimizes rotation along $\alpha$, i.e., $\nabla_{\alpha'}\xi$ is a multiple of $\alpha'$ (see Corollary 1 of Ref. \cite{daSilvaMJM2018}). In this work, we add the important observation that the quantity $\lambda$ above could be identified with the mean curvature $H$ of the geodesic sphere (see Theorems \ref{thr::NtuIFFxiRM} and \ref{thr::SpcFrmViaTUspheresUsingRM}). This follows from the equality $\nabla_{\alpha'}\xi=-\lambda\alpha'$ for every spherical curve $\alpha$, which means that all principal curvatures are equal and, then, $H=\lambda$. In addition, this reasoning provides an alternative proof, in terms of RM frames, for the well known result
\newline
\textbf{Theorem A}\label{thr::GeoSphInCteCurvAreTotalUmb}
\textit{ Every geodesic sphere in $\mathbb{R}^{m+1}$, $\mathbb{H}^{m+1}(r)$, and  $\mathbb{S}^{m+1}(r)$, is a totally umbilical hypersurface.}

The reader may consult \cite{Spivak1979v4}, or \cite{VanheckePRSocEdinburgh1979}, for proofs of the above theorem via other techniques. (When $\mathbb{H}^{m+1}(r)$ and  $\mathbb{S}^{m+1}(r)$ are modeled as hyperspheres in Euclidean and Lorentzian spaces, respectively, their geodesic spheres can be fully described through intersections with hyperplanes: see, e.g., Figs. 1(b) and 1(c) of \cite{daSilvaMJM2018}.) We remark that the {converse} of Theorem A is also valid:
\newline
\newline
\textbf{Theorem B (Kulkarni, Vanhecke, Willmore, and Chen \cite{kulkarniPAMS1975,VanheckePRSocEdinburgh1979,ChenCrelle1981})} \label{thr::CharSpcFrmViTotalUmbGeoSph}
\textit{Let $M^{m+1}$ be an $(m+1)$-dimensional connected Riemannian manifold ($m\geq 2$). Then, $M^{m+1}$ is a space form if, and only if, every sufficiently small geodesic sphere in $M^{m+1}$ is totally umbilical.} 

Notice that in general we have to restrict ourselves to work with sufficiently small radii because out of the injectivity radius geodesic spheres may fail to be properly defined or to be a smooth hypersurface. On the other hand, for $\mathbb{R}^{m+1}$, $\mathbb{H}^{m+1}(r)$, and  $\mathbb{S}^{m+1}(r)$ no restriction has to be {imposed} at all since the exponential map in such ambient spaces is always globally defined.

Now, taking into account that, according to the sign of the sectional curvature, every space form is locally isometric to $\mathbb{H}^{m+1}(r)$, $\mathbb{R}^{m+1}$, or  $\mathbb{S}^{m+1}(r)$ \cite{Spivak1979v4}, the Theorems A and B above suggest the following problem:
\newline
\newline
\textbf{ {Main Problem:}}
Assume that every curve on a geodesic sphere $G(p,R)$ of a connected manifold $M^{m+1}$ satisfies Eq. (\ref{eq::SimpleCurvesEq}), with $\lambda=H$ the mean curvature of $G$ and $R$ sufficiently small. Does $M^{m+1}$ have constant sectional curvature?
\newline

In this work we shall prove that {the Main Problem} is answered in the affirmative. The strategy consists in first showing that: (i) the unit normal of a hypersurface $\Sigma^m\subset M^{m+1}$ along a curve $\alpha:I\to \Sigma$ minimizes rotation if, and only if, $\alpha$ satisfies Eq. (\ref{eq::SimpleCurvesEq}) with $\lambda=H$ the mean curvature of $\Sigma$; and (ii) a submanifold $\Sigma^m\subset M^{m+1}$ is totally umbilical if, and only if, condition (i) is valid for every curve $\alpha:I\to \Sigma$. These two results together {give} a characterization of Riemannian space forms, in terms of spherical curves, as a corollary of Theorem B.  In addition, we also show the possibility of characterizing Riemannian space forms via the total torsion of closed spherical curves and also in terms of an inequality involving the curvature function of a curve and the mean curvature of the hypersurface where the curve lies. Finally, we show that these same results can be extended to semi-Riemannian geometry.

The {remainder} of this work is divided as follows. In Sect. 2 we present some background material. In Sect. 3 we characterize totally umbilical hypersurfaces via the concept of RM frames, provide a new proof for Theorem B above, and characterize Riemannian space forms using their spherical curves. In the two following sections, we provide  characterizations of totally umbilical hypersurfaces and space forms via an inequality involving the mean curvature of a hypersurface and the curvature function of their curves (Sect. 4) and via the total torsion of closed spherical curves in the 3-dimensional case (Sect. 5). Finally, in Sect. 6, we discuss the extension of the results presented in Sections 3 and 4 to the semi-Riemannian setting.

\section{Differential geometric background}

Let $M^{m+1}$ be an $(m+1)$-dimensional Riemannian manifold of class at least {$C^3$} with metric $\langle\cdot,\cdot\rangle$, Levi--Civita connection $\nabla$, and Riemann curvature tensor
\begin{equation}
R(X,Y)Z=\nabla_Y\nabla_XZ-\nabla_X\nabla_YZ+\nabla_{[X,Y]}Z,
\end{equation}
where $X,Y$, and $Z$ are tangent vector fields in $M$. In addition, if $X,Y\in T_pM$, the \textit{sectional curvature} of $\mbox{span}\{X,Y\}\subset T_pM$ at $p$ is 
\begin{equation}
K_p(X,Y)=\frac{\langle R(X,Y)X,Y\rangle_p}{\langle X,X\rangle_p\langle Y,Y\rangle_p-\langle X,Y\rangle_p^2}.
\end{equation}
We say that $M$ is a \textit{space form} if it has constant sectional curvature $K$. In addition, $M$ is locally isometric to a sphere $\mathbb{S}^{m+1}(r)$ if $K=r^{-2}$, to an Euclidean space $\mathbb{E}^{m+1}$ if $K=0$, and to a hyperbolic space $\mathbb{H}^{m+1}(r)$ if $K=-r^{-2}$  \cite{doCarmo1992,Spivak1979v4}.

A $C^k$ curve $\alpha:I\to M^{m+1}$ is \textit{regular} if $\mathbf{t}(s)=\alpha'(s)\not=0$, where $s$ denotes arc-length parameter, i.e., $\langle\mathbf{t}(s),\mathbf{t}(s)\rangle=1$. Usually, we equip a regular curve with the well know Frenet frame. However, as shown by Bishop in the 1970s \cite{BishopMonthly}, we can also consider another type of an orthonormal moving frame  with many useful geometric properties. We say that a unit vector field $\mathbf{x}\in\mathfrak{X}(M)$ normal to $\alpha$ is a \textit{Rotation Minimizing} (RM) \textit{vector field} along $\alpha:I\to M^{m+1}$ if $\nabla_{\mathbf{t}}\,\mathbf{x}=\lambda\,\mathbf{t}$ for some function $\lambda$, where $\mathfrak{X}(M)$ denotes the module of tangent vector fields of $M$. (This concept corresponds to a parallel transport with respect to the normal connection of the curve $\alpha$ \cite{Etayo2016}.) Thus, we say that $\{\mathbf{t},\mathbf{n}_1,\dots,\mathbf{n}_m\}$ is an \textit{RM frame} along $\alpha$ if each normal vector field $\mathbf{n}_i$ is RM. (The basic idea here is that $\mathbf{n}_i$ rotates only the necessary amount to remain normal to $\mathbf{t}=\alpha'$.) The equations of motion are then given by
\begin{equation}
\nabla_{\mathbf{t}}\,\mathbf{t}=\kappa_1\mathbf{n}_1+\dots+\kappa_m\mathbf{n}_m\mbox{ and }\nabla_{\mathbf{t}}\,\mathbf{n}_i=-\kappa_i\mathbf{t},\,i=1,\dots,m.\label{eq::RMEqs}
\end{equation}

\begin{example}
Let $\alpha:I\to M$ be a geodesic. Given a unit  vector $\mathbf{x}_0$  normal to $\mathbf{t}(s_0)$, let $\mathbf{x}$ be the parallel transport of $\mathbf{x}_0$ {along} $\alpha$. Since parallel transport preserves angles, $\mathbf{x}$ is a normal vector field along $\alpha$ and, in addition, since $\nabla_{\mathbf{t}}\,\mathbf{x}=0$, $\mathbf{x}$ minimizes rotation. In short, any orthonormal and parallel transported frame along a geodesic is RM. This also shows that, contrarily to the Frenet frame, an RM frame can be defined even if the curvature $\kappa=\Vert\nabla_{\mathbf{t}}\,\mathbf{t}\Vert$ of a curve $\alpha$ vanishes.  (It is worth mentioning that RM frames are uniquely defined only up to a rotation on the {hyperplane} normal to the curve  \cite{BishopMonthly}.)
\qed
\end{example}

Rotation minimizing frames are particularly useful in the characterization of spherical curves. Indeed, Bishop showed that an Euclidean curve is spherical if, and only if, there exist constants $a_1,\dots,a_m$ such that  $
\sum_{i=1}^ma_i\kappa_i(s)=1$ \cite{BishopMonthly}. (Bishop's equation can be rewritten in a more convenient form if we normalize the constants {$a_i$} by the radius of the sphere: the new constants can be then interpreted as the coordinates of the unit normal with respect to the basis $\{\mathbf{n}_i\}_{i=1}^m $, while the linear coefficient is the mean curvature, see Eq. (\ref{eq::SimpleCurvesEq}).) The same characterization for spherical curves was recently extended to hyperbolic spaces and spheres \cite{daSilvaMJM2018}. Here   we shall prove this is in fact a characteristic feature of totally umbilical hypersurfaces. 

The \textit{exponential map} at $p\in M$ in the direction of a unit vector $V\in T_pM$ is $\exp_p(uV)=\beta_V(u)$, where $u$ is sufficiently small and $\beta_{V}$ is the unique geodesic with initial conditions $\beta_V(0) =p$ and $\beta_V'(0)=V$. The exponential map defines a diffeomorphism of a neighborhood of $0\in T_pM$ into a neighborhood of $p\in M$. For a sufficiently small $R>0$, we may define in $M$ the \textit{geodesic sphere}  with center $p$ and radius $R$ as the submanifold $G(p,R)=\exp_p(\mathbb{S}^m(R))$, where $\mathbb{S}^m(R)=\{R\,V\in T_pM:\Vert V\Vert=1\}$. ($G(p,R)$ is a smooth submanifold for $R$ sufficiently small.) A natural choice for the unit normal $\xi$ of $G(p,R)$ at $q=\beta_{V}(R)$ is the vector  $\xi(q)=\beta_{V}'(R)$, i.e., the {tangent} at $u=R$ of the radial geodesic emanating from $p$. In the next section, we shall prove that $\xi$ is RM if, and only if, $M$ is a space form (see Theorems \ref{thr::NtuIFFxiRM} and \ref{thr::CharSpcFrmsViaSphrclCrv}).

Let $\Sigma^m\subset M^{m+1}$ be an orientable hypersurface with unit normal $\xi$. The \textit{shape operator}  of $\Sigma$ is defined by
\begin{equation}
S_p(X)=-(\nabla_X\xi)(p).
\end{equation}
Since $S_p$ is symmetric with respect to $\langle\cdot,\cdot\rangle$, in general, $S_p$ has $m$ eigenvalues $\lambda_1,\dots,\lambda_m$, known as the \textit{principal curvatures}.  A point $p\in \Sigma^m$ is \textit{umbilical} if every unit vector $X\in T_p\Sigma^m$ is an eigenvector of the shape operator $S_p$. When every point of $\Sigma^m$ is umbilical, $\Sigma^m$ is said to be a \textit{totally umbilical hypersurface}. A curve $\alpha:I\to \Sigma^m$ is a \textit{line of curvature} if for every $s\in I$, the tangent $\alpha'(s)$ is a \textit{principal direction}, i.e., the \textit{normal curvature} $\kappa_n=\langle \nabla_{\alpha'}\,\alpha',\xi\rangle=\langle S(\alpha'),\alpha'\rangle$ is an eigenvalue of $S$ along  the points of $\alpha$. The \textit{mean curvature} of $\Sigma^m$ at the point $p$ is then defined as $H=\frac{1}{m}(\lambda_1+\dots+\lambda_m)$.

\section{Rotation minimizing frames and Riemannian space forms}

The characterization of spherical curves presented in Ref. \cite{daSilvaMJM2018} has to do with the fact that geodesic spheres in those spaces are totally umbilical. Indeed, here we establish a characterization of totally umbilical hypersurfaces in terms of RM frames and, when restricted to geodesic spheres, it allows for a characterization of space forms. First, we shall characterize lines of curvature.

\begin{lemma}\label{lemma::RMnormalIFFlineCurv}
Let $\alpha:I\to \Sigma^m$ be a  $C^2$ regular curve in an orientable hypersurface $\Sigma^m\subset M^{m+1}$ {with unit normal $\xi$}. The following conditions are equivalent
\begin{enumerate}
\item $\xi$ is a rotation minimizing vector field along $\alpha$;
\item $\alpha$ is a line of curvature;
\item if we write $\xi(\alpha(s))=\sum_{i=1}^ma_i\mathbf{n}_i(s)$ along $\alpha$, then 
\begin{equation}
\displaystyle\sum_{i=1}^ma_i\kappa_i(s) = \kappa_n(\alpha(s))\mbox{ and } a_1,\dots,a_m\mbox{ are constants},\label{eq::NormalDevolopmentSphrCurves}
\end{equation}
where $\kappa_n$ denotes the normal curvature and $\kappa_i$ the $i$-th curvature associated with a given rotation minimizing frame $\{\alpha',\mathbf{n}_1,...,\mathbf{n}_m\}$ along $\alpha(s)$.
\end{enumerate}
\end{lemma}
\begin{proof}
($1\Leftrightarrow 2$) The unit normal $\xi$ of $\Sigma$ is RM along $\alpha$ if and only if
\begin{equation}
S_{\alpha(s)}(\alpha'(s))=-(\nabla_{\alpha'}\,\xi)(\alpha(s)) = \lambda(s) \,\alpha'(s).
\end{equation}
Then, we conclude that $\alpha'(s)$ is an eigenvector of $S$ for every $s$ if and only if $\xi$ is RM along $\alpha$. In addition, $\lambda(s)$ is precisely the normal curvature in the direction of $\alpha'(s)$, which is then a principal curvature. In short, $\xi$ is RM along $\alpha$ if and only if $\alpha$ is a line of curvature.

($2\Rightarrow 3$) Let $\{\alpha',\mathbf{n}_1,...,\mathbf{n}_m\}$ be an RM frame along $\alpha$. Since $\langle \alpha',\xi\rangle=0$, we have $\xi=\sum_{i=1}^ma_i\mathbf{n}_i$ along $\alpha$. Furthermore, taking the derivative of $\langle \alpha',\xi\rangle=0$, and using $\nabla_{\alpha'}\xi=-\kappa_n \alpha'$, gives
\begin{eqnarray}
0 & = & \langle \nabla_{\alpha'}\alpha', \xi\rangle + \langle \alpha',\nabla_{\alpha'}\xi\rangle\nonumber\\
& = & \langle\sum_i\kappa_i\mathbf{n}_i,\sum_ja_j\mathbf{n}_j\rangle-\kappa_n= \sum_{i=1}^ma_i\kappa_i-\kappa_n.
\end{eqnarray}
Now, taking the derivative of $a_i=\langle \xi,\mathbf{n}_i\rangle$ gives
\begin{equation}
a_i' = \langle\nabla_{\alpha'}\xi,\mathbf{n}_i\rangle+\langle \xi,\nabla_{\alpha'}\mathbf{n}_i\rangle= -\kappa_n\langle \alpha',\mathbf{n}_i\rangle-\kappa_i\langle \xi,\alpha'\rangle= 0
\end{equation}
and, therefore, the coefficients $a_1,\dots,a_m$ are constants. 

($3\Rightarrow2$) Let Eq. (\ref{eq::NormalDevolopmentSphrCurves}) be valid. The unit normal along $\alpha$ is then written as $\xi(s)=\sum_ia_i\mathbf{n}_i(s)$. Taking the derivative gives
\begin{equation}
\nabla_{\alpha'} \xi = \sum_{i=1}^ma_i\nabla_{\alpha'}\mathbf{n}_i= \left(-\sum_{i=1}^ma_i\kappa_i\right) \alpha'=-\kappa_n \alpha'
\end{equation}
and, thus, $\xi$ is RM along $\alpha:I\to \Sigma^m$ and $\alpha'$ is a principal direction. 
\qed
\end{proof}

\begin{theorem}\label{thr::NtuIFFxiRM}
An orientable hypersurface $\Sigma^m\subset M^{m+1}$ is totally umbilical if and only if {its} unit normal $\xi$ minimizes rotation along any regular curve $\alpha:I\to \Sigma$. In addition, Eq. (\ref{eq::NormalDevolopmentSphrCurves}), which is satisfied by every curve in $\Sigma,$ is rewritten as 
\begin{equation}
\sum_{i=1}^ma_i\kappa_i(s)=H(\alpha(s)),
\end{equation}
where $H$ is the mean curvature of $\Sigma.$
\end{theorem}
\begin{proof}
The total umbilicity means that every tangent vector is an eigenvector of  $S(X)=-\nabla_X\xi$, i.e., every curve is a line of curvature. Consequently, the normal curvature $\kappa_n(q,X)$ only depends on $q$ and not on $X$, and then 
\begin{equation}
H(q)=\frac{1}{m}\sum_{i=1}^m\kappa_n({q})=\kappa_n(q).
\end{equation}
The desired result then follows from Lemma \ref{lemma::RMnormalIFFlineCurv}.   
\qed
\end{proof}

We now apply the concept of minimizing rotation to provide a simple proof for Theorem B which, to the best of our knowledge, was first proved by Kulkarni \cite{kulkarniPAMS1975}. (See also \cite{ChenCrelle1981} and \cite{VanheckePRSocEdinburgh1979} for alternative demonstrations.)

\begin{theorem}\label{thr::SpcFrmViaTUspheresUsingRM}
Let $M^{m+1}$ be a connected Riemannian manifold ($m\geq 2$). Then, $M^{m+1}$ is a space form if, and only if, every sufficiently small geodesic sphere in $M^{m+1}$ is totally umbilical. 
\end{theorem}
\begin{proof}
Assume all sufficiently small geodesic spheres  in $M$ are totally umbilical  or, equivalently, that the unit normal is RM along any geodesic spherical curve. Given unit vectors $X_q,Y_q\in T_qM$ at $q$, we can choose  $p\in M$ sufficiently close to $q$ such that  the geodesic sphere $G(p,R)$ passing at $q$ is tangent to $X_q$ and normal to $Y_q$. Notice that there exists a unit speed curve $V:(-\varepsilon,\varepsilon)\to \mathbb{S}^{m}(1)\subset T_pM$ such that  the 2-surface $f(u,s)=\exp_p(uV(s))$ leads to
\begin{equation}
q=f(R,0),\,\frac{\partial f}{\partial s}(R,0)=X_q,\mbox{ and }\frac{\partial f}{\partial u}(R,0)=Y_q.
\end{equation}
We may extend $X_q$ and $Y_q$ to $X=\partial/\partial s$ and  $Y=\partial/\partial u$, respectively. Then $\nabla_{Y}Y=0$ and, in addition, $[X,Y]=0$, {which implies $\nabla_XY=\nabla_YX$.} From the umbilicity condition, {i.e.,} $\nabla_XY=-\lambda X$, it follows 
\begin{eqnarray}
R(Y,X)Y &=& \nabla_X\nabla_YY-\nabla_Y\nabla_XY+\nabla_{[Y,X]}Y\nonumber\\
& = & \nabla_Y(\lambda X)= \left(\frac{\partial\lambda}{\partial u}-\lambda^2\right)X.
\end{eqnarray}
The sectional curvature in the direction of $\mbox{span}\{X_q,Y_q\}$ at $q$ is then given by
\begin{equation}
K_q(X,Y)=\frac{\langle R(Y,X)Y,X\rangle}{\langle X,X\rangle\langle Y,Y\rangle-\langle X,Y\rangle^2}=\left(\frac{\partial\lambda}{\partial u}-\lambda^2\right)\Big\vert_{u=R},
\end{equation}
which only depends on $q$. Therefore, from the arbitrariness in the choice of $X_q,Y_q$, the sectional curvature $K$ is a function of the base point only and, then, $M$ is an isotropic manifold. Finally, by the Schur theorem \cite{doCarmo1992}, the sectional curvature should be a constant.

{Conversely}, let $M$ be a space form. Since the proportionality between $X$ and $\nabla_X\xi$  does not change under isometries and any space form $M$ is locally isometric to $\mathbb{H}^{m+1}(r)$, $\mathbb{R}^{m+1}$, or $\mathbb{S}^{m+1}(r)$ \cite{doCarmo1992,Spivak1979v4}, we may restrict ourselves to this particular context. Now, using that the unit normal $\xi$ along a spherical curve $\alpha(s)=\exp_p(uV(cs))$  may be written as $\xi(s)=\frac{\partial}{\partial u}\exp_p(uV(cs))\vert_{u=R}$, for some convenient unit speed curve $V:I\to \mathbb{S}^m(1)\subset T_pM$ {and constant $c$}, it follows from Corollary 1 of Ref. \cite{daSilvaMJM2018} that $\xi$ is RM along any spherical curve and, consequently, from Theorem \ref{thr::NtuIFFxiRM} we conclude that all geodesic spheres in $\mathbb{H}^{m+1}(r)$, $\mathbb{R}^{m+1}$, and $\mathbb{S}^{m+1}(r)$ are totally umbilical.
\qed
\end{proof}

Since Riemannian space forms are characterized by the property that every geodesic sphere is totally umbilical, our main result  follows as a corollary of Theorems \ref{thr::NtuIFFxiRM}  and \ref{thr::SpcFrmViaTUspheresUsingRM}.
\begin{theorem}\label{thr::CharSpcFrmsViaSphrclCrv}
Let $M^{m+1}$ be a connected Riemannian manifold ($m\geq 2$). Then, $M$ is a space form if, and only if, for every regular $C^2$ curve $\alpha$ on a geodesic sphere, with sufficiently small radius, there exist constants $a_i$ such that
\begin{equation}
\sum_{i=1}^ma_i\kappa_i(s)=H(s),\,\xi=\sum_{i=1}^ma_i\,\mathbf{n}_i,
\end{equation}
where $H$ is the mean curvature of the geodesic sphere which contains $\alpha$, $\kappa_i$ are the curvatures associated with a rotation minimizing frame $\{\mathbf{t},\mathbf{n}_1,\dots,\mathbf{n}_m\}$ along $\alpha$, and $\xi$ is a  {unit vector field normal} to the geodesic sphere. 
\end{theorem}

\begin{remark}
In \cite{daSilvaMJM2018} it was shown that Frenet frames can {be} used to characterize geodesic spherical curves in $\mathbb{S}^{m+1}(r)$ and $\mathbb{H}^{m+1}(r)$: e.g., a $C^4$ regular curve  $\alpha:I\to\mathbb{S}^3(r)$, or $\mathbb{H}^3(r)$, lies on a sphere if, and only if, $(\frac{1}{\tau}(\frac{1}{\kappa})')'+\frac{\tau}{\kappa}=0$, where  $\kappa$ and $\tau$ denote the curvature and torsion, respectively. (Theorem 4 of \cite{daSilvaMJM2018}.)  A Frenet frame $\{\alpha',\mathbf{n},\mathbf{b}\}$ can  be also used to investigate RM unit normals. {Indeed}, if $\Sigma^2\subset M^3$ is totally umbilical, its normal $\xi$ is RM along any $\alpha:I\to \Sigma^2$, i.e., $-\nabla_{\alpha'}\xi=\lambda\alpha'$. Writing $\xi(s)=c_1(s)\mathbf{n}(s)+c_2(s)\mathbf{b}(s)$ gives
\begin{eqnarray}
\nabla_{\alpha'}\xi & = & c_1'\mathbf{n}+c_2'\mathbf{b}+c_1\nabla_{\alpha'}\mathbf{n}+c_2\nabla_{\alpha'}\mathbf{b}\\
-\lambda\alpha' & = & -c_1\kappa\,\alpha'+(c_1'-\tau c_2)\mathbf{n}+(c_2'+\tau c_1)\mathbf{b}.
\end{eqnarray}
We then conclude that $c_1=\lambda/\kappa$, $c_1'=c_2\tau$, and $c_2'=-\tau c_1$, which lead to
\begin{equation}\label{eq::SphCrvODEviaKappaTau}
\frac{\mathrm{d}}{\mathrm{d}s}\left[\frac{1}{\tau}\frac{\mathrm{d}}{\mathrm{d}s}\left(\frac{\lambda}{\kappa}\right)\right]+\tau\,\frac{\lambda}{\kappa}=0.
\end{equation}
{Conversely, if Eq. \eqref{eq::SphCrvODEviaKappaTau} is valid, then $\xi=\frac{\lambda}{\kappa}\mathbf{n}+\frac{1}{\tau}(\frac{\lambda}{\kappa})'\,\mathbf{b}$ satisfies $\nabla_{\alpha'}\xi=-\lambda \alpha'$, and also $\frac{\mathrm{d}}{\mathrm{d} s} \langle\xi,\xi\rangle=\frac{\mathrm{d}}{\mathrm{d} s}\{(\frac{\lambda}{\kappa})^2
+[\frac{1}{\tau}(\frac{\lambda}{\kappa})']^2\}=0$. Thus, the normal to $\Sigma$ is RM along all $\alpha$ in $\Sigma$.} {Note that geodesic} spheres on a space form are totally umbilical and, in addition, they have constant mean curvature \cite{Spivak1979v4}. In such cases, $\lambda$ is a constant and can be canceled {out in Eq. \eqref{eq::SphCrvODEviaKappaTau}}, as in Theorem 4 of \cite{daSilvaMJM2018}.
\end{remark}

\section{Curvature of geodesic spherical curves}

We now present a characterization of Riemannian space forms based on an inequality involving the curvature $\kappa=\Vert\nabla_{\alpha'}\alpha'\Vert$ of $\alpha:I\to \Sigma^m\subset M^{m+1} $ and the mean curvature $H$ of $\Sigma^m$. 

In \cite{KimMonthly2003} Baek \textit{et al.} proved that it is possible to characterize Euclidean spheres in terms of the curvature function of spherical curves:
\newline
\textbf{Theorem C (Baek, Kim, and Kim \cite{KimMonthly2003})}\label{thr::CharacEuclSpheresKoreanos}
\textit{Let $r>0$ be a constant and $\Sigma^m\subset\mathbb{R}^{m+1}$ be a closed hypersurface such that for every $\,C^2$ regular curve $\alpha:I\to \Sigma^m$ it is valid}  
\begin{equation}
\kappa \geq \frac{1}{r},
\end{equation}
\textit{where $\kappa$ is the curvature function of $\alpha$ in $\mathbb{R}^{m+1}$ and  equality holds for geodesics only. Then, $\Sigma^m$ is a sphere of radius $r$.}
\newline

This result has to do with the fact that spheres in Euclidean space are totally umbilical. (By letting $r\to\infty$ we can include planes in the criterion above.) A crucial observation to extend this theorem to Riemannian manifolds is that $1/r$ may be replaced by the mean curvature of $\Sigma^m\subset M^{m+1}$. Indeed,   
\begin{theorem}\label{thr::CharOfTUviaKappaAndH}
Let $\Sigma^m\subset M^{m+1}$ be an orientable hypersurface. Then, $\Sigma^m$ is totally umbilical if, and only if, for every regular $C^2$ curve $\alpha:I\to \Sigma^m$ it is valid the inequality
\begin{equation}
\kappa(s) \geq \vert H(\alpha(s)) \vert,\label{eq::CharOfTUviaKappaAndH}
\end{equation}
where $H$ is the mean curvature of $\Sigma^m$ and $\kappa$ is the curvature function of $\alpha$ in $M^{m+1}$. In addition, equality only holds when $\alpha$ is a geodesic. 
\end{theorem}
\begin{proof}
Any regular curve $\alpha(s)$ in $\Sigma^m$ leads to
\begin{eqnarray}
\kappa^2 = \kappa_g^2+\kappa_ n^2,\label{eq::KappaKappaGandKappaN}
\end{eqnarray}
where $\kappa_g$ and $\kappa_n$ are the geodesic and normal curvatures of $\alpha$ in $\Sigma$, respectively. 

If $\Sigma$ is totally umbilical, then $H(p)=\kappa_n(p)$, since $\kappa_n$ is a function of $p\in \Sigma$ only. Substituting this in the equation above gives $\kappa^2\geq H^2$ and, then, $\kappa\geq\vert H\vert$. In addition, $\kappa^2=H^2\Leftrightarrow\kappa_g\equiv0\Leftrightarrow\alpha:I\to \Sigma$ is a geodesic.

Conversely, let $\lambda_1,\dots,\lambda_m$ be the principal curvatures of $\Sigma$ at $p$ and $\alpha_i$ be the geodesic tangent to the $i$-th principal direction at $p$, i.e., $\kappa_n(p,\alpha'_i(0))=\lambda_i$, $\alpha_i(0)=p$. It follows from the hypothesis, and from Eq. (\ref{eq::KappaKappaGandKappaN}), that $\forall\,i\in\{1,\dots,m\},\,H(p)^2=\lambda_i^2$. Since $H$ is the arithmetic mean of the $\lambda_i$'s, the principal curvatures are all equal and, consequently, $\Sigma$ is totally umbilical.
\qed
\end{proof}

Now, applying 
Theorems { \ref{thr::SpcFrmViaTUspheresUsingRM}  and} \ref{thr::CharOfTUviaKappaAndH} to geodesic spheres, we have another characterization of space forms.
\begin{theorem}
Let $M^{m+1}$ ($m\geq2$) be a connected Riemannian manifold. Then, $M^{m+1}$ is a space form if, and only if, every curve on a geodesic sphere of sufficiently small radius satisfies Eq. (\ref{eq::CharOfTUviaKappaAndH}), with equality valid for geodesics {only}.
\end{theorem}

\section{Total torsion of closed geodesic spherical curves}

We shall now restrict ourselves to 3-dimensional manifolds and define the \textit{total torsion}, $T$, of a $C^3$ regular curve $\alpha:I\to M^3$ as the integral $T=\int_a^b\tau(s)\mathrm{d}s$. It is known that the total torsion $T$ of any \textit{closed} spherical curve in Euclidean space vanishes (the same is trivially true for plane curves). {Conversely}, if the total torsion of every closed curve on a surface ${\Sigma^2}\subset\mathbb{R}^3$ vanishes, than ${\Sigma}$ is a plane or a sphere \cite{scherrer1940kennzeichnungKugel}, see also \cite{geppert1941carattSfera}. In other words, vanishing total torsion characterizes totally umbilical surfaces in $\mathbb{R}^3$, i.e., spheres and planes. Recently, the investigation of the total torsion of closed curves on a totally umbilical surface were extended to Riemannian space forms \cite{PansonatoGD2008totalTorsion}. (Notice that Pansonato and Costa call spherical curves any curve on a totally umbilical surface. Here, we shall only apply this terminology to curves that really lie on a geodesic sphere.)
\newline
\textbf{Theorem D (Pansonato and Costa \cite{PansonatoGD2008totalTorsion})}
\textit{Let $\Sigma^2$ be a connected surface on a three-dimensional space form $M^3$, then $\Sigma^2$ is totally umbilical if, and only if,  every closed curve $\alpha:I\to \Sigma^2$ has a vanishing total torsion.}  

\begin{remark}
Interestingly, the above characterization remains valid if we replace the total torsion, $\oint\tau$, by an integral $\oint f(\kappa)\tau$ with $f$ continuous \cite{PansonatoGD2008totalTorsion,yin2017curvature}.
\end{remark}

In this context, the Darboux frame plays an important role. We may equip $\alpha:I\to \Sigma^2$ with the \textit{Darboux frame} $\{\mathbf{t}=\alpha',\mathbf{h},\xi\}$, where $\{\mathbf{t}(s),\mathbf{h}(s)\}$ is a positive basis for $T_{\alpha(s)}\Sigma$ and $\xi$ the surface normal. The equation of motion is
\begin{equation}
\nabla_{\mathbf{t}}\left(
\begin{array}{c}
\mathbf{t}\\
\mathbf{h}\\
\xi\\
\end{array}
\right)=\left(
\begin{array}{ccc}
0 & \kappa_g & \kappa_n\\
-\kappa_g & 0 & \tau_g\\
 -\kappa_n & -\tau_g & 0\\ 
\end{array}
\right)\left(
\begin{array}{c}
\mathbf{t}\\
\mathbf{h}\\
\xi\\
\end{array}
\right),\label{eq::RiemDarbouxEqs}
\end{equation}
where $\kappa_g$ and $\kappa_n$ are  the \textit{geodesic} and \textit{normal curvatures} of $\alpha$ in $\Sigma$, respectively, while $\tau_g$ is the \textit{geodesic torsion}. Notice that $\xi$ minimizes rotation if, and only if, $\tau_g$ vanishes. In particular, lines of curvature are given by $\tau_g=0$ (Lemma \ref{lemma::RMnormalIFFlineCurv}) and, consequently, $\Sigma^2\subset M^3$ is a totally umbilical surface if, and only if, $\tau_g$ vanishes identically {for every curve in $\Sigma^2$}. Denoting by $\theta$ the angle between the principal normal $\mathbf{n}$ and the surface normal $\xi$, gives \cite{PansonatoGD2008totalTorsion}
\begin{equation}
\tau_g(s)=\tau(s)+\theta'(s).\label{eq::AngleBetweenFrenetDarboux}
\end{equation}

Proceeding with the guiding philosophy of this work, we now furnish a characterization of space forms in terms of the total torsion of closed curves. To do that, we first need to remove the restriction on the ambient manifold in the Theorem 3 of \cite{PansonatoGD2008totalTorsion}, where one assumes $M$ admits a conformal parameterization. This can be accomplished by imposing that the surface $\Sigma$ is simply connected.

\begin{lemma}\label{lem::CharTUsurfacesViaTotalTorsion}
Let $\Sigma^2\subset M^3$ be an orientable and simply connected surface on a three-dimensional connected Riemannian manifold $M^3$. If every $C^3$ closed curve $\alpha:I\to \Sigma^2$ has vanishing total torsion, then $\Sigma^2$ is totally umbilical. 
\end{lemma}
\begin{proof}
Let $\alpha:I\to \Sigma$ be a closed curve. Using the relation between the Frenet and Darboux frames in Eq. (\ref{eq::AngleBetweenFrenetDarboux}), we have
\begin{equation}
0=\oint_{\alpha}\tau=\oint_{\alpha}\tau_g-\oint_{\alpha}\theta'\Rightarrow \oint_{\alpha}\tau_g=2n\pi,\,n\in\mathbb{Z},\label{eq::TotalGeodTorsionForClosedCurvWithZeroTotalTorsion}
\end{equation}
where we used that $\int\theta'$ should be an integer multiple of $2\pi$ since the surface normal $\xi$ and the curve principal normal $\mathbf{n}$ return to their initial position (notice $\alpha^{(i)}(s_o)=\alpha^{(i)}(s_f)$, $i=1,2,3$). 

The Eq. (\ref{eq::TotalGeodTorsionForClosedCurvWithZeroTotalTorsion}) implies that $\oint_{\alpha}\tau_g$ does not vary under $C^3$ deformations of $\alpha$. Since $\Sigma^2$ is simply connected, we can deform $\alpha$ and make it as close as we want to a sufficiently small quadrilateral in $\Sigma^2$ whose sides are lines of curvature. Since $\tau_g$ vanishes for lines of curvature, we find $n=0$ and, then,
\begin{equation}
\forall\,\alpha: {I} \to \Sigma^2,\,\oint_{\alpha}\tau_g=\oint_{\alpha}\langle-\nabla_{\alpha'}\,\xi,(\alpha')^{\perp}\rangle=0,\label{eq::ZeroTotalGeodTorsionForClosedCurvWithZeroTotalTorsion}
\end{equation}
where $\{\alpha'(s),(\alpha')^{\perp}(s)\}$ is a positive orthonormal basis of $T_{\alpha(s)}\Sigma^2$. From Eq. (\ref{eq::ZeroTotalGeodTorsionForClosedCurvWithZeroTotalTorsion}) it follows that for every $\alpha'(s)\in T_{\alpha(s)}\Sigma$ we must have $\nabla_{\alpha'}\xi\perp (\alpha')^{\perp}$. Indeed, if it were $\langle(\nabla_{X}\,\xi)(p),X^{\perp}(p)\rangle\not=0$, $X\in\mathfrak{X}(\Sigma)$,  then by continuity every closed curve around $p$ would have $\tau_g>0$ (or $\tau_g<0$) and, then, $\oint\tau_g\not=0$. This contradicts Eq. (\ref{eq::ZeroTotalGeodTorsionForClosedCurvWithZeroTotalTorsion}). In other words,  $\nabla_{\alpha'}\xi$ should be a multiple of $\alpha'$, i.e., $\xi$ minimizes rotation. The lemma then follows from Theorem \ref{thr::NtuIFFxiRM}.
\qed
\end{proof}

\begin{theorem}\label{thr::CharSpcFrmTotalTorsion}
Let $M^3$ be a {three-dimensional} connected Riemannian manifold. Then, $M^3$ is a space form if, and only if, every closed curve $\alpha:I\to G(p,R)$ on a geodesic sphere, with a sufficiently small radius $R$, has a vanishing total torsion. 
\end{theorem}
\begin{proof}
Since geodesic spheres on a space form are totally umbilical, the ``only if'' part is already proved in Theorem 2 of \cite{PansonatoGD2008totalTorsion}. The converse follows from {lemma \ref{lem::CharTUsurfacesViaTotalTorsion}}, since geodesic spheres are simply connected.
\qed
\end{proof}

\section{Semi-Riemannian space forms}

One may naturally ask whether the characterizations of manifolds with constant sectional curvature given here would remain valid in semi-Riemannian spaces. We shall see below that this can be effectively done. As can be verified in the previous sections, see e.g. proof of Lemma \ref{lemma::RMnormalIFFlineCurv}, the diagonalizability of the shape operator was a crucial feature to employ RM vector fields as a tool. In  semi-Riemannian geometry too the shape operator is diagonalizable at an umbilical point and  every non-lightlike direction is a principal direction (a vector $V\not=0$ on a semi-Riemannian manifold $(M_{\nu}^{m+1},\langle\cdot,\cdot\rangle_{\nu})$ with  index $\nu$ may have a \textit{causal character}, i.e., it is \textit{timelike} if $\langle V,V\rangle_{\nu}<0$, \textit{lightlike} if $\langle V,V\rangle_{\nu}=0$, and \textit{spacelike} if $\langle V,V\rangle_{\nu}>0$ \cite{ONeil}).  

It often happens in semi-Riemannian geometry that one should exclude lightlike vectors/hyperplanes when defining geometric quantities. This is the case for the sectional curvature, which is not defined for a lightlike plane \cite{ONeil}, p. 229. In the following we only work with non-lightlike curves and hypersurfaces having the same causal characterer in all their points ($\alpha(t)$ is \textit{lightlike} when $\alpha'(t)$ is lightlike for all $t$, i.e., $\alpha'$ is normal to itself; and $\Sigma^m$ is \textit{lightlike} if the induced metric is degenerate).

A normal vector field $N$ \textit{minimizes rotation} along a non-lightlike curve $\alpha$ when $\nabla_{\alpha'}N$ and $\alpha'$ are parallel. Then, any RM frame $\{\alpha',\mathbf{n}_1,\dots,\mathbf{n}_m\}$ along a non-lightlike curve $\alpha$ satisfies equations of motion similar to those found in Lorentz-Minkowski space $\mathbb{E}_1^3$ \cite{daSilvaJG2017}:
\begin{equation}
\nabla_{\alpha'}\alpha'=\sum_{i=1}^m\varepsilon_i\kappa_i\mathbf{n}_i\mbox{ and }\nabla_{\alpha'}\mathbf{n}_i=\varepsilon\kappa_i\alpha',
\end{equation}
where $\varepsilon_i=\langle\mathbf{n}_i,\mathbf{n}_i\rangle_{\nu}=\pm1$ and $\varepsilon=\langle\alpha',\alpha'\rangle_{\nu}=\pm1$, i.e., $+1$ for spacelike vectors and $-1$ for timelike ones. {Now, we may extend previous Riemannian results to our new context. The semi-Riemannian version of Theorem \ref{thr::NtuIFFxiRM} is given by} 

\begin{theorem}\label{thrSemiRie::NtuIFFxiRM}
A non-lightlike and orientable hypersurface $\Sigma^m\subset M_{\nu}^{m+1}$ is totally umbilical if, and only if, {its} unit normal $\xi$ minimizes rotation along any non-lightlike regular curve $\alpha:I\to \Sigma^m$. In addition, along $\alpha$ it is valid the equation
\begin{equation}
\sum_{i=1}^m\varepsilon\, a_i\kappa_i(s)=\eta\, H(\alpha(s)),
\end{equation}
where $H$ is the mean curvature of $\Sigma$, $\eta=\langle\xi,\xi\rangle_{\nu}=\pm1$, and $\varepsilon=\langle\alpha',\alpha'\rangle_{\nu}=\pm1$.
\end{theorem}
\begin{proof}
The proof is similar to that of Theorem \ref{thr::NtuIFFxiRM}. We should only pay some attention on the causal characters of the vector fields {$\xi$ and $\alpha'$. Indeed}, if $\xi$ is RM along any non-lightlike curve $\alpha$, then we can write $\xi=\sum_ia_i\mathbf{n}_i$ for some constants $a_1,\dots,a_m$ and, in addition,
\begin{equation}
\nabla_{\alpha'}\xi = \sum_ia_i\nabla_{\alpha'}\mathbf{n}_i=-\sum_ia_i\varepsilon\kappa_i\alpha'\Rightarrow \kappa_n(\alpha')=\varepsilon\sum_ia_i\kappa_i.
\end{equation}
The desired result then follows, since the mean curvature is $H=\eta\left(\sum_i\lambda_i\right)/m$, where the $\lambda_i$'s are the eigenvalues of the shape operator $-\nabla_X\xi$. 
\qed
\end{proof}
{In addition, the semi-Riemannian version of Theorem  \ref{thr::CharOfTUviaKappaAndH} is given by}

\begin{theorem}\label{thr::CharOfLorentzianTUviaKappaAndH}
A non-lightlike and orientable hypersurface $\Sigma^m\subset M_{\nu}^{m+1}$ is totally umbilical if, and only if, for every non-lightlike regular $C^2$ curve $\alpha:I\to \Sigma^m$ one has
\begin{equation}
\left\{
\begin{array}{l}
\langle\alpha'',\alpha''\rangle_{\nu} \geq \eta H^2,\,\mbox{ if }\,\nabla^{\Sigma}_{\alpha'}\alpha'\mbox{ is not timelike}\\[5pt]
\langle\alpha'',\alpha''\rangle_{\nu} \leq \eta H^2,\,\mbox{ if }\,\nabla^{\Sigma}_{\alpha'}\alpha'\mbox{ is not spacelike}\\
\end{array}
\right.,
\label{eq::CharOfLorTUviaKappaAndH}
\end{equation}
where $H=H(s)$ is the mean curvature of $\Sigma$ {along $\alpha(s)$}, $\eta=\langle\xi,\xi\rangle_{\nu}=\pm1$, $\alpha''=\nabla_{\alpha'}\alpha'$, and $\nabla^{\Sigma}$ is the induced connection on $\Sigma$. In addition, equality only holds when $\alpha$ is a geodesic. 
\end{theorem}
\begin{proof}
The proof is analogous to that of Theorem \ref{thr::CharOfTUviaKappaAndH} once we {notice} that 
\begin{equation}
\alpha''=\nabla_{\alpha'}^{\Sigma}\alpha'+\kappa_n\xi\Rightarrow \langle\alpha'',\alpha''\rangle_{\nu}=\langle\nabla_{\alpha'}^{\Sigma}\alpha',\nabla_{\alpha'}^{\Sigma}\alpha'\rangle_{\nu}+\kappa_n^2\langle\xi,\xi\rangle_{\nu}.
\end{equation}
\qed
\end{proof}

\begin{remark}
Recently, Theorem \ref{thr::CharacEuclSpheresKoreanos} was extended to spacelike spheres in $\mathbb{E}_1^{m+1}$ \cite{KimBKMS2018}, i.e., the hyperbolic space in the hyperboloid model. Since in this context one has $\langle\alpha'',\alpha''\rangle=\kappa_g^2-\kappa_n^2$, a proof similar to that of Theorem \ref{thr::CharOfTUviaKappaAndH} can be provided to characterize totally umbilical spacelike hypersurfaces via the inequality $\langle\alpha'',\alpha''\rangle \geq - H^2$. This is a particular instance of Theorem \ref{thr::CharOfLorentzianTUviaKappaAndH} above.
\end{remark}

We are finally able to provide a characterization of semi-Riemannian manifolds with constant sectional curvature by applying to geodesic spheres the two above results concerning totally umbilical hypersurfaces. Similarly to what happens in Riemannian geometry,  semi-Riemannian space forms can be characterized by the umbilicity of their small geodesic spheres. Indeed, we have

\begin{theorem}
Let $M_{\nu}^{m+1}$ be a connected semi-Riemannian manifold ($m\geq 2$). Then, $M^{m+1}$ is a space form if, and only if, every sufficiently small geodesic sphere in $M_{\nu}^{m+1}$ is totally umbilical. 
\end{theorem}
\begin{proof}
If every sufficiently small geodesic sphere in $M_{\nu}^{m+1}$ is totally umbilical, then a proof similar to that of Theorem \ref{thr::SpcFrmViaTUspheresUsingRM} shows that $M_{\nu}^{m+1}$ is a space form. 

Conversely, let $M_{\nu}^{m+1}$ be a semi-Riemannian manifold with constant sectional curvature $K$. Then, $M$ is locally isometric to a hyperquadric of $\mathbb{E}_{\nu}^{m+2}$: $\mathbb{S}_{\nu}^{m+1}(r)$ if $K=r^{-2}$; $\mathbb{E}_{\nu}^{m+1}$ if $K=0$; and $\mathbb{H}_{\nu}^{m+1}(r)$ if $K=-r^{-2}$ \cite{ONeil}. As in the Riemannian case, here we do the proof for this particular setting without any loss of generality. 

For $\mathbb{E}_{\nu}^{m+1}$, any hyperquadric $Q$ of radius $R$ with position vector $q$ has $\xi=q/R$ as the unit normal. Then, any curve $\alpha:I\to Q$ satisfies  $\nabla_{\alpha'}\xi=\alpha'/R$, which means that $\xi$ is RM along any curve and, therefore, $Q$ is totally umbilical (see \cite{daSilvaJG2017}, and references therein, for a further investigation  of this fact in $\mathbb{E}_1^3$). In the other cases, the normal to a geodesic sphere $G(p,R)$ can be built from the tangents to the geodesics emanating from the center of the sphere, i.e., along a curve $\alpha(s)=\exp_p(RV(s))$ the unit normal is $\xi=\frac{\partial}{\partial u}\exp_p(uV(s))\vert_{u=R}$ (see the analogous construction on $\mathbb{S}^{m+1}(r)$ and $\mathbb{H}^{m+1}(r)$ in Ref.  \cite{daSilvaMJM2018}).

A geodesic $\beta(u)$ in  $\mathbb{S}_{\nu}^{m+1}(r)=\{p\in\mathbb{E}_{\nu}^{m+2}:\langle p,p\rangle_{\nu}=r^2\}$ with initial condition $(p,V)\in T\mathbb{S}_{\nu}^{m+1}(r)$ may be written as
\begin{equation}\label{eq::SpcGeodInSnu}
\beta(u) =\exp_p(uV)= \cos(\frac{u}{r})p+r\sin(\frac{u}{r})V,\,\langle V,V\rangle_{\nu}=+1,
\end{equation} 
if $\beta$ is spacelike, or as
\begin{equation}\label{eq::TimeGeodInSnu}
\beta(u) =\exp_p(uV)= \cosh(\frac{u}{r})p+r\sinh(\frac{u}{r})V,\,\langle V,V\rangle_{\nu}=-1,
\end{equation}
if $\beta$ is timelike. It follows that a curve $\alpha:I\to G(p,R)$ may be written as $\alpha(s)=\exp_p(RV(cs))$, where $V$ is a unit speed curve in $T_p\mathbb{S}_{\nu}^{m+1}(r)$ with (i) $c=[r\sin(R/r)]^{-1}$ and $\langle V',V'\rangle_{\nu}=+1$, if $\alpha$ is spacelike and (ii) $c=[r\sinh(R/r)]^{-1}$ and $\langle V',V'\rangle_{\nu}=-1$, if $\alpha$ is timelike. Notice, Eq. (\ref{eq::SpcGeodInSnu}) [or Eq. (\ref{eq::TimeGeodInSnu})] {should} be used {when} the unit normal $\xi$ of $G(p,R)$ is spacelike (or timelike, respectively). 

Now we prove that $\xi(s)=\frac{\partial}{\partial u}\exp_p(uV(cs))\vert_{u=R}$ is RM along $\alpha(s)=\exp_p(RV(cs))$. If $\xi$ is spacelike, then $\alpha'(s)= c\,r\sin(\frac{R}{r})V'(cs)=V'(cs)$ and
\begin{eqnarray}
\nabla_{\alpha'}\xi & = & \nabla_{\alpha'}\left[-\frac{1}{r}\sin(\frac{R}{r})p+\cos(\frac{R}{r})V(cs)\right]\nonumber\\
&=& c\,\cos(\frac{R}{r})V'(cs)=\frac{1}{r}\cot(\frac{R}{r})\,\alpha'(s),
\end{eqnarray}
where we used that for orthogonal fields $X,Y\in \mathfrak{X}(\mathbb{S}_{\nu}^{m+1}(r))$ the (induced) covariant derivative $\nabla_XY$ in  $\mathbb{S}_{\nu}^{m+1}(r)$ coincides with the usual covariant derivative $D_XY$ in $\mathbb{E}_{\nu}^{m+1}$: $(D_XY)(p)=(Y\circ \gamma)'(0)$, $\gamma(0)=p$ and $\gamma'(0)=X$. On the other hand, if $\xi$ timelike, we may analogously find that 
\begin{equation}
\nabla_{\alpha'}\xi=\frac{1}{r}\coth(\frac{R}{r})\,\alpha'.
\end{equation}
In short, the unit normal of any geodesic sphere in $\mathbb{S}_{\nu}^{m+1}(r)$ minimizes rotation and, consequently, the geodesic spheres of $\mathbb{S}_{\nu}^{m+1}(r)$ are totally umbilical.

Finally, similar computations can be performed on geodesic spheres of the space form of negative curvature $\mathbb{H}_{\nu}^{m+1}(r)=\{p\in\mathbb{E}_{\nu+1}^{m+2}:\langle p,p\rangle_{\nu+1}=-r^2\}$.
\qed
\end{proof}

\begin{remark}
Notice, we did not consider an extension of Theorem \ref{thr::CharSpcFrmTotalTorsion}. We believe this is a subtle subject, since a semi-Riemannian hypersurface may fail to have a closed  curve with the same causal character  in all its point: e.g., in $\mathbb{E}_1^3$ there is no closed timelike curve \cite{LopezIEJG2014}. In this respect, it would be necessary to take {into} account curves with a changing causal character and, therefore, also understand what happens with the curve torsion near lightlike points.
\end{remark}

\begin{acknowledgements}
The Authors would like to {thank} \'Alvaro Ramos (Universidade Federal do Rio Grande do Sul) and Fernando Etayo (Universidad de Cantabria) for useful discussions. The first author, LCBdS, would also like to thank the staff from the Departamento de Matem\'atica, Universidade Federal de Pernambuco (Recife, Brazil) where part of this research was done while the author worked as a temporary lecturer at that institution. 
\end{acknowledgements}



\end{document}